\documentclass{amsproc}

\usepackage{graphicx}
\usepackage{fancyhdr}
\usepackage{epsfig}
\usepackage{color}
\usepackage{xy}

\newtheorem{theorem}{Theorem}[section]
\newtheorem{lemma}[theorem]{Lemma}

\theoremstyle{definition}
\newtheorem{definition}[theorem]{Definition}

\newtheorem{proposition}[theorem]{Proposition}
\newtheorem{corollary}[theorem]{Corollary}

\theoremstyle{remark}
\newtheorem{remark}[theorem]{Remark}

\def\N{{\mathbb N}} %N, Q, R und C als K"orperzeichen bold
\def\Z{{\mathbb Z}}

\def\R{{\mathbb R}}
\def\C{{\mathbb C}}

\def\Im{\mathrm{Im}\,}
\def\Re{\mathrm{Re}\,}
\def\supp{\mathrm{supp}\,}
\def\Span{\mathrm{span}\,}

\def\skpl{\langle}      % Skalarprodukt links
\def\skpr{\rangle}      % Skalarprodukt rechts

\newcommand{\mcF}{\mathcal{F}}
\newcommand{\mcD}{\mathcal{D}}
\newcommand{\st}{\,:\,}
\newcommand{\be}{\begin{equation}}
\newcommand{\ee}{\end{equation}}

\newcommand{\bfa}{\boldsymbol{a}}
\newcommand{\bfz}{\boldsymbol{z}}

\newcommand{\Arg}{{\mathrm{Arg\,}}}
\newcommand{\loc}{{\mathrm{loc}}}
\newcommand{\inn}[2]{{\langle #1,#2 \rangle}}

\begin{document}

\title{Exponential Splines of Complex Order}

\author{Peter Massopust}
\address{Institute of Biomathematics and Biometry, Helmholtz Zentrum M\"unchen\\ Ingolst\"adter Landstrasse 1\\ 85764 Neuherberg, Germany and Zentrum Mathematik\\Lehrstuhl M6\\ Technische Universtit\"at M\"unchen \\ 
Boltzmannstrasse 3\\ 85747 Garching b. M\"unchen\\ Germany}
%\curraddr{Department of Mathematics and Statistics,
%Case Western Reserve University, Cleveland, Ohio 43403}
\email{massopust@ma.tum.de; peter.massopust@helmholtz-muenchen.de}
%%    Information for second author
%\author{Author Two}
%\address{Mathematical Research Section, School of Mathematical Sciences,
%Australian National University, Canberra ACT 2601, Australia}

%    General info
%\date{January 1, 1994 and, in revised form, June 22, 1994.}


\begin{abstract}
We extend the concept of exponential B-spline to complex orders. This extension contains as special cases the class of exponential splines and also the class of polynomial B-splines of complex order. We derive a time domain representation of a complex exponential B-spline depending on a single parameter and establish a connection to fractional differential operators defined on Lizorkin spaces. Moreover, we prove that complex exponential splines give rise to multiresolution analyses of $L^2(\R)$ and define wavelet bases for $L^2(\R)$.
\end{abstract}
\maketitle
\section{Preliminaries on Exponential Splines}\label{sec2}
Exponential splines are used to model phenomena that follow differential 
systems of the form $\dot{x} = A x$, where $A$ is a constant matrix. For 
such equations the coordinates of the the solutions are linear 
combinations of functions of the type $e^{a x}$ and $x^n e^{a x}$, $a\in \R$. In 
approximation theory exponential splines are modeling data that exhibit 
sudden growth or decay and for which polynomials are ill-suited because 
of their oscillatory behavior. Some of the mathematical issues regarding 
exponential splines in the theory of interpolation and approximation can 
be found in the following references: 
\cite{ammar,dm1,dm2,mccartin,sakai1, sakai2, spaeth,unserblu05,zoppou}.

Another approach to exponential splines is based on certain classes of
linear differential operators with constant coefficients. The original
ideas of such an approach can be found in, for instance, \cite{mic,unserblu05}
and in exposition in \cite{masso}. The classical
polynomial splines $s$ of order $n$, $n\in \N$, can be interpreted as 
(distributional) solutions to equations of the form
\begin{equation}\label{poly}
D^n s = \sum_{\ell =1}^n c_\ell\,\delta (\cdot - \ell), \quad c_\ell\in \R.
\end{equation}
where $D$ denotes the (distributional) derivative and $\delta$ the
Dirac delta distribution. A well known class of splines of central 
importance is the class of  polynomial B-splines $B_n$, defined as the 
$n$-fold convolution of  the characteristic function of the unit interval. Polynomial B-splines satisfy the above equation and they lay the foundations for further generalizations.

Equation \eqref{poly} is a special case of the more general expression
\begin{equation}\label{exp}
L^n f := \prod_{j=1}^n (D + a_j I) f = \sum_{\ell=1}^n c_\ell\,\delta (\cdot - \ell), \quad a_j, c_\ell\in \R.
\end{equation}
where $I$ denotes the identity operator. Solutions to (\ref{exp}) are then called {\em exponential splines} and they reduce to polynomial splines in case all $a_j = 0$. For later purposes, we record a particular identity involving a special case of \eqref{exp}: If, for all $j\in\{1, \ldots, n\}$, $a_j =: a\in \R$, then
\be\label{important}
(D + aI)^n (e^{-a(\bullet)}f) = e^{-a(\bullet)} D^n f,\qquad n\in \N.
\ee
One can {\em define} the differential operator on the left-hand side in this manner and show that this definition is equivalent to the usual definition involving the binomial theorem for linear differential operators with constant coefficients:
\[
(D + aI)^n = \sum_{k=0}^n \binom{n}{k} a^k D^{n-k}.
\]
However, for our later purposes of replacing the integer $n$ by a complex number $z$, such a {\em finite} expression is not available and we need to resort to \eqref{important} as the basic identity.

In this paper, we extend the concept of exponential spline to include complex orders in the defining equation \eqref{exp}. For this purpose, we need to extend the differential operator $L^n$ to a fractional differential operator $\mathcal{L}^z$ of complex order $z$ defined on an appropriate function space. We obtain the generalization of exponential splines to complex order via exponential B-splines. To this end, we briefly review polynomial and exponential B-splines, and revisit the definition of polynomial B-splines of complex order.  The former is done in Section \ref{sec2} and the latter in Section \ref{sec4}. In Section \ref{sec4}, we also introduce the fractional derivative operators and function spaces needed for the generalization. Exponential splines of complex order depending on one parameter are then defined in Section \ref{sec5}. For this purpose, we first introduce exponential B-splines of complex order, for short complex exponential B-splines, in the Fourier domain and discuss some of their properties. In particular, we derive a time domain representation and show that this new class of splines defines multiresolution analyses of and wavelet bases for $L^2(\R)$. At this point, we also establish the connection to fractional differential operators of complex order defined on Lizorkin spaces. A brief discussion of how to incorporate more than one parameter into the definition of a complex exponential B-spline and the derivation of an explicit formula in the time domain for complex exponential B-splines depending on two parameters concludes this section. The last section summarizes the results and describes future work. 
\section{Brief Review of Polynomial and Exponential B-Splines}\label{sec2}
Based on the interpretation (\ref{exp}) one defines, analogously to the introduction of B-splines, exponential B-splines as convolution products of exponential functions $e^{a (\cdot)}$ restricted to $[0,1]$. In this section, we briefly review the definitions of polynomial and exponential B-splines. For more details, we refer the interested reader to the vast literature on spline theory.
\subsection{Polynomial B-Splines}
Let $n\in \N$. The $n$th order classical Curry-Schoenberg (polynomial) B-splines \cite{curryschoenberg} are defined as the $n$-fold convolution product of the characteristic function $\chi = \chi_{[0,1]}$ of the unit interval:
\[
B_{n} := \underset{j = 1}{\overset{n}{*}}\chi.
\]
Equivalently, one may define $B_n$ in the Fourier domain as
\[
\mathcal{F}(B_n)(\omega ) =:  \widehat{B_n} (\omega) := \int_{\R} B_n(x) \,e^{-i\omega x}\, dx = \left( \frac{1-e^{-i\omega}}{i\omega}\right)^{\,n},
\]
where $\mcF$ denotes the Fourier--Plancherel transform on $L^2(\R)$.

Polynomial B-splines generate a {\em discrete} family of approximation/interpolation functions with increasing smoothness:
\[
B_n \in C^{\,n-2}, \quad n\in \N.
\]
and possess a natural multiscale structure via knot insertions. (Here, $C^{-1}$ is interpreted as the space of piece-wise continuous functions.) In addition, polynomial B-splines generate approximation spaces and satisfy several recursion relations that allow fast and efficient computations within these spaces.
\subsection{Exponential B-Splines}
Exponential B-splines of order $n\in \N$ are defined as $n$-fold
convolution products of exponential functions of the form $e^{a (\cdot)}$ restricted
to the interval $[0,1].$ More precisely, let $n\in \N$ and $\boldsymbol{a}:=(a_1, \ldots, a_n)\in \R^n$, with at least one $a_j\neq 0$. Then the {\em exponential B-splines of order $n$ for the $n$-tuple of parameters $\bfa$} is given by
\be\label{regE}
E_n^{\bfa} :=  \underset{j = 1}{\overset{n}{*}} \left(e^{a_j (\bullet)}\chi\right).
\ee
This class of splines shares several properties with the classical polynomial B-splines, but there are also significant differences that makes them useful for different purposes. In \cite{CM}, a useful explicit formula for
these functions was derived and those cases characterized for which the integer translates of an exponential B-spline form a
partition of unity up to a multiplicative factor, i.e.,
\[
\sum_{k\in \Z} {E}^{\bfa}_n(x-k)=C, \ x\in \R,
\] 
for some constant $C.$ 

Moreover, series expansions for functions in  $L^2(\R)$ in terms of shifted and modulated versions of exponential B-splines were derived, and dual pairs of Gabor frames based on exponential B-splines constructed. We note that exponential B-splines also have been used to construct multiresolution analyses and obtain wavelet expansions. (See, for instance,  \cite{unserblu05,LY}.) In addition, it is shown in \cite{CS} that exponential splines play an important role in setting up a one-to-one correspondence between dual pairs of Gabor frames and dual pairs of wavelet frames.
\section{Polynomial Splines of Complex Order}\label{sec4}
Now, we like to extend the concept of cardinal polynomial B-splines to orders other than $n\in\N$. Such an extension to real orders was investigated in \cite{unserblu00,zheludev}. In \cite{unserblu05}, these splines were named {\em fractional B-splines}. Their extension to complex orders was undertaken in \cite{forster06}. The resulting class of cardinal B-splines of complex order or, for short, {\em complex B-splines}, $B_z: \R \to\C$ are defined in the Fourier domain by
\begin{equation}
\mathcal{F}(B_z)(\omega ) =:  \widehat{B_z} (\omega) := \int_{\R} B_z(t)e^{-i\omega t}\, dt := \left( \frac{1-e^{-i\omega}}{i\omega}\right)^z,
\label{eq Definition Fourier B-Spline}
\end{equation}
for $z\in \C_{>1} := \{\zeta\in \C\st \Re \zeta >1\}$. At the origin, there exists a continuous continuation satisfying $\widehat{B_z}(0)=1$. Note that as $\{\frac{1-e^{-i\omega}}{i\omega}\mid \omega\in \R\} \cap \{y \in\R \mid y<0\} = \emptyset$, complex B-splines reside on the main branch of the complex logarithm and are thus well-defined. 

The motivation behind the definition of complex B-splines is the need for a single-band frequency analysis. For some applications, e.g., for phase retrieval tasks, complex-valued analysis bases are needed since real-valued bases can only provide a symmetric spectrum. Complex B-splines combine the advantages of spline approximation with an approximate one-sided frequency analysis. In fact, the spectrum $|\widehat B_{z}(\omega)| $ has the following form. Let
\be\label{Omega}
\Omega(\omega) := \frac{1-e^{-i\omega}}{i \omega}.
\ee
Then the spectrum consists of the spectrum of a real-valued B-spline, combined with a modulating and a damping factor:
$$
|\widehat{B_z}(\omega)| = |\widehat{B_{\Re z}}(\omega)| e^{-i \Im z \ln |\Omega (\omega)|} e^{\Im z \arg \Omega(\omega)}.
$$
The presence of the imaginary part $\Im z$ causes the frequency components on the negative and positive real axis to be  enhanced with different signs. This has the effect of shifting the frequency spectrum towards the negative or positive frequency side, depending on the sign of $\Im z$. The corresponding bases can be interpreted as approximate single-band filters \cite{forster06}.

For the purposes of this article, we summarize some of the most important properties of complex B-splines. Complex B-splines have a time-domain representation of the form
\be\label{cb}
B_z(x) = \frac{1}{\Gamma(z)} \sum_{k= 0}^\infty (-1)^k \left( {z} \atop {k}\right) (x-k)_+^{z-1},
\ee
where the equality holds point-wise for all $x\in\R$ and in the $L^2(\R)$--norm \cite{forster06}. Here, the complex valued binomial is defined by
\[
{z \choose k} := \frac{\Gamma (z+1)}{\Gamma (k+1) \Gamma (z - k +1)},
\]
where $\Gamma :\C\setminus\Z_0^-\to \C$ denotes the Euler Gamma function,
and
\[
x_+^z := \begin{cases}
x^z = e^{z \ln x},      & x > 0;\\
    0,  &  x \leq 0,
\end{cases}
\]
is the complex-valued truncated power function. Formula \eqref{cb} can be verified by Fourier inversion of \eqref{eq Definition Fourier B-Spline}.

Equation \eqref{cb} shows that $B_z$ is a piecewise polynomial of complex degree $z-1$ and that its support is, in general, not compact. It was shown in \cite{forster06} that $B_z$ belongs to the Sobolev spaces $W^s(L^2 (\R))$ for $\Re z > s + \frac{1}{2}$ (with respect to the $L^2$-norm and with weight $(1+|x|^2)^s$). The smoothness of the Fourier transform of $B_z$ implies fast decay in the time domain:
$$
B_z (x) \in \mathcal{O}(x^{-m}), \quad \mbox{for $m < \Re z +1$, as $\ |x| \to \infty$}. 
$$
If $z, z_1, z_2\in \C_{>1}$, then the convolution relation 
$$
B_{z_1} \ast B_{z_2} = B_{z_1+z_2}
$$
and the recursion relation 
\[
B_z (x) = \frac{x}{z-1}\,B_{z-1}(x) + \frac{z-x}{z-1}\,B_{z-1} (x - 1)
\]
holds. Complex B-splines generate a {\em continuous} family of approximation/interpo\-lation functions in the sense that they are elements of (inhomogenous) H\"older spaces \cite{unserblu00}:
\[
B_{z} \in C^s, \quad{s := \Re z - 1}, \quad z\in \C_{>1}.
\]
In addition, complex B-splines are scaling functions, i.e., they satisfy a two-scale refinement equation, generate multiresolution analyses and wavelets, and relate difference and differential operators. For more details and other properties of this new class of splines, we refer the interested reader to \cite{statisticalencounters,FGMS,FMUe,forstermasso10,brigittesampta,hassip,multivarSplines,spie,masso,Mmoments,petersampta}.

Based on the more general definition \cite{karlin,masso,unserblu05} of polynomial splines $s :[a,b]\to \R$ of order $n$ as the solution of a differential equation of the form 
\[
D^n s = \sum_{\ell=0}^n c_\ell \delta {x - \ell},\quad c_\ell\in \R,
\]
where $D$ denotes the distributional derivative and $\delta_{x_\nu}$ the Dirac distribution at $\ell\in \Z$, one can define a spline of complex order $z$ in a similar manner \cite{forstermasso10}. For this purpose, we denote by $\mathcal{S}(\R)$ the Schwartz space of rapidly decreasing functions on $\R$, and introduce the {\em Lizorkin space}
\[
\Psi := \left\{\psi\in \mathcal{S}(\R)\st D^m \psi ({0}) = 0, \,\forall m\in \N\right\},
\]
and its restriction to the nonnegative real axis:
\[
\Psi_+ := \{f\in \Psi \st \supp f \subseteq [0,\infty)\}.
\]
Let $\C_+ := \{z\in \C\st \Re z > 0\}$ and define a kernel function $K_z:\R\to\C$ by
\[
K_z(x) := \frac{x_+^{z-1}}{\Gamma (z)}.
\]
For an $f\in \Psi_+$, define a fractional derivative operator $\mathcal{D}^z$ of complex order $z$ on $\Psi_+$ by
\be\label{diffz}
\mcD^z f := \underbrace{(D^n f) * K_{n-z}}_{\footnotesize{\textrm{(Caputo)}}} = \underbrace{D^n (f*K_{n-z})}_{\footnotesize{\textrm{(Riemann--Liouville)}}}, \qquad n = \lceil \Re z \rceil.
\ee
where $*$ denotes the convolution on $\Psi_+$. Note that, since we are defining the fractional derivative operator on the Lizorkin space $\Psi_+$, the Caputo and Riemann-Liouville fractional derivatives coincide.

Similarly, a fractional integral operator $\mcD^{-z}$ of complex order $z$ on $\Psi_+$ is defined by
\[
\mcD^{-z} f := f*K_z.
\]
It follows from the definitions of $\mcD^{z}$ and $\mcD^{-z}$ that $f\in \Psi_+$ implies $\mcD^{\pm z} f \in \Psi_+$ and that all derivatives of $\mcD^{\pm z} f$ vanish at $x = 0$. The convolution-based definition of the fractional derivative and integral operator of functions $f\in \Psi_+$ also ensures that $\{\mcD^z \st z\in \C\}$ is a semi-group in the sense that
\begin{equation}
\mcD^{z + \zeta} = \mcD^z \mcD^\zeta = \mcD^\zeta \mcD^z = \mcD^{\zeta +z},
\end{equation}
for all $z,\zeta\in \C$. (See, for instance, \cite{podlubny}.)

For our later purposes, we need to define fractional derivative and integral operators $\mcD^{\pm z}$ on the dual space $\Psi_+^\prime$ of $\Psi_+$. To this end, we may regard the locally integrable function $K_z$, $\Re z > -1$, as an element of $\Psi_+^\prime$ by setting
\[
\skpl K_z, \varphi\skpr = \int_0^\infty K_z (t) \varphi (t) dt, \quad \forall\varphi\in \Psi_+.
\]
Here, $\skpl\bullet, \bullet\skpr$ denotes the canonical pairing between $\Psi_+$ and $\Psi_+^\prime$. In passing, we like to mention that the function $z\mapsto \skpl K_z, \varphi\skpr$, $\varphi\in \Psi_+$, is holomorphic for all $z\in\C\setminus \N_0$.
\begin{remark}
The function $K_z$ may also be defined for general $z\in \C$ via Hadamard's partie finie and represents then a pseudo-function. For more details, we refer the interested reader to \cite{dantray,gelfand}, or \cite{zemanian}.
\end{remark}

Note that for $f,g\in \Psi_+^\prime$ the convolution exists on $\Psi_+^\prime$ and is defined in the usual way \cite{samko} by 
\be\label{eq7}
\skpl f*g, \varphi\skpr :=  \skpl (f\times g) (x,y), \varphi (x+y)\skpr = \skpl f(x), \skpl g(y), \varphi (x+y)\skpr\skpr, \quad \varphi\in \Psi_+.
\ee
The pair $(\Psi_+^\prime, *)$ is a convolution algebra with the Dirac delta distribution $\delta$ as its unit element. Thus, we can extend the operators $\mcD^{\pm z}$ to $\Psi_+^\prime$ in the following way.

Let  $z\in \C_+$, let $\varphi\in \Psi_+$ be a test function and $f\in \Psi_+^\prime$. Then the fractional derivative operator $\mcD^{z}$ on $\Psi_+^\prime$ is defined by
\[
\skpl\mcD^{z} f, \varphi\skpr := \skpl (D^nf)*K_{n-z},\varphi\skpr,\quad n = \lceil\Re z\rceil,
\]
and the fractional integral operator $\mcD^{-z}$ by $\skpl\mcD^{-z} f, \varphi\skpr := \skpl f*K_z,\varphi\skpr$. The semi-group properties (\ref{important}) also hold for $f\in \Psi_+^\prime$. (For a proof, see \cite{podlubny} or \cite{gelfand}.)

In \cite{podlubny, samko,gelfand} it was shown that the $z$-th derivative of a truncated power function is given 
\be\label{dertruncpow}
\mcD^z \left[\frac{(x - k)_+^{z-1}}{\Gamma (z)}\right] = \delta (x - k), \quad k < x \in [0, \infty).
\ee
Thus, by the semi-group properties of $\mcD^z$, one obtains $\mcD^{-z} \delta (\bullet - k) = \frac{(\bullet - k)^{z-1}_+}{\Gamma (z)}.$

Now, we are ready to define a spline of complex order $z$ \cite{forstermasso10}: Let $z\in \C_+$ and let $\{a_k\st k\in \N_0\}\in \ell^\infty (\R)$. A solution of the equation
\be\label{cs}
\mcD^z f = \sum_{k=0}^\infty a_k\, \delta (\bullet - k)
\ee
is called a \textit{spline of complex order $z$}.

It can be shown \cite{forstermasso10} that the complex B-spline
\[
B_z (x) = \frac{1}{\Gamma(z)} \sum_{k=0}^\infty (-1)^k \binom{z}{k} (x-k)_+^{z-1}, \quad z\in \C_{>1}.
\]
is a solution of Equation (\ref{cs}) with
\[
a_k = (-1)^k \binom{z}{k},
\]
and is thus a nontrivial example of a spline of complex order.
\section{Exponential Splines of Complex Order}\label{sec5}
In this section, we to extend the concept of exponential B-spline to incorporate complex orders while maintaining the favorable properties of the classical exponential splines. 
\subsection{Definition and basic properties}
To this end, we take the Fourier transform of an exponential function of the form $e^{ -a x}$, $a\in \R$, and define in complete analogy to \eqref{eq Definition Fourier B-Spline}, an {\em exponential B-spline of complex order $z\in \C_{>1}$ for $a\in \R$} (for short, complex exponential B-spline) in the Fourier domain by
\be\label{expspline}
\widehat{E_z^a} (\omega):=\left(\frac{1-e^{-(a+i\omega)}}{a+i\omega}\right)^z.
\ee
We set
\[
\Omega(\omega, a) := \frac{1-e^{-(a+i\omega)}}{a+i\omega},
\]
and note that trivially $\Omega(\omega,0) = \Omega (\omega)$ and, therefore, $\widehat{E^0_z} = \widehat{B_z}$; see \eqref{Omega}.

The function $\Omega(\bullet, a)$ only well-defined for $a \geq 0$. We may verify this as follows. The real part and imaginary parts of $\Omega(\bullet, a)$ are explicitly given by
\begin{align*}
\Re\Omega(\omega, a) := f(\omega, a) & = \frac{e^{-a} \omega  \sin \omega -e^{-a} a \cos\omega +a}{a^2+\omega ^2},\\
\Im\Omega(\omega, a) := g(\omega, a) & = \frac{a e^{-a} \sin\omega +e^{-a} \omega  \cos\omega -\omega }{a^2+\omega ^2}.
\end{align*}
If $g(\omega,a) = 0$ then $a\sin\omega = e^a\omega - \omega\cos\omega$. Therefore, 
\begin{align*}
f(\omega,a) & = \frac{e^{-a}}{a(a^2+\omega^2)}\left(a \omega \sin\omega - a^2\cos\omega + a^2e^a\right)\\
& = \frac{e^{-a}}{a(a^2+\omega^2)}\left(e^a\omega^2 - \omega^2\cos\omega - a^2\cos\omega + a^2e^a\right)\\
& = \frac{e^{-a}}{a}\left(e^a - \cos\omega\right) \geq 0, \quad\text{only if $a \geq 0$}.
\end{align*}
Thus, the graph of $\Omega(\omega, a)$ does not cross the negative $x$-axis; see Figure \ref{fig1} for examples reflecting the different choices for $a$.
\begin{figure}[h!]
\includegraphics[width = 3cm, height = 3cm]{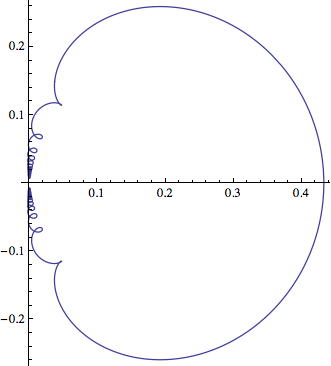}\hspace{2cm}
\includegraphics[width = 3cm, height = 3cm]{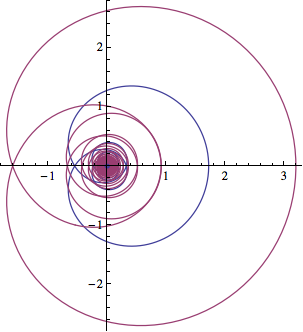}
\caption{The graph of $\Omega(\bullet, a)$: $a = 2$ (left) and $a = -1, -2$ (right).}\label{fig1}
\end{figure}

In particular, this implies that the infinite series
\[
\Omega(\omega,a)^z = \sum_{\ell=0}^\infty \binom{z}{\ell} (-1)^\ell \frac{e^{-(a+ i \omega)\ell}}{(a+i\omega)^z}
\]
converges absolutely for all $\omega\in \R$.
 
As a side result, which is summarized in the following proposition, we obtain the asymptotic behavior of $\Omega(\bullet, a)$ as $a\to\infty$.
\begin{proposition}
The real and imaginary parts of $\Omega(\bullet, a)$ satisfy the identity
\[
\left(f(\omega,a) - \frac{1}{2a}\right)^2 + g(\omega,a)^2 =  \frac{1}{4 a^2} + \frac{e^{-a} \left(-\frac{\omega  \sin (\omega
   )}{a}+e^{-a}-\cos (\omega )\right)}{a^2+\omega
   ^2}, \quad \omega\in \R.
\] 
Furthermore, the curve $\mathcal{C}(a) := \{(f(\omega,a),g(\omega,a)) \st \omega\in \R\}$ approaches the circle
\[
K(a):\quad\left(x - \frac{1}{2a}\right)^2 + y^2 = \frac{1}{4 a^2},
\]
in the sense that
\[
|\mathcal{C}(a) - K(a)| \in e^{-2 a}\mathcal{O}(a^{-2}) + e^{-a}\mathcal{O}(a^{-2}), \quad a \gg 1.
\]
\end{proposition}
\begin{proof}
The first statement is a straight-forward algebraic verification, and the second statement follows from the linearization of $\frac{e^{-a} \left(-\frac{\omega  \sin (\omega)}{a}+e^{-a}-\cos (\omega )\right)}{a^2+\omega^2}$.
\end{proof}

\begin{remark}
For real $z > 0$, the function $\Omega(\omega)^z$ and its time domain representation were already investigated in \cite{westphal} in connection with fractional powers of operators and later also in \cite{unserblu00} in the context of extending Schoenberg's polynomial splines to real orders. In the former, a proof that this function is in $L^1 (0,\infty)$ was given using arguments from summability theory (cf. Lemma 2 in \cite{westphal}), and in the latter the same result was shown but with a different proof. In addition, it was proved in \cite{unserblu00} that for real $z$, $\Omega(\omega)^z\in L^2 (\R)$ for $z > 1/2$ (using our notation). (Cf. Theorem 3.2 in \cite{unserblu00}.)
\end{remark}

To obtain a relationship between $\Omega (\bullet, a)$ and $\Omega$, we require a lemma whose straightforward proof is omitted.
\begin{lemma}\label{lem5.3}
For all $x\in \R$, we have the following inequalities between $\cos$ and $\cosh$.
\[
\frac{1-\cos x}{x^2} \leq \frac12\leq\frac{\cosh x - 1}{x^2}.
\]
\end{lemma}

Our next goal is to obtain inequalities relating $\Omega (\bullet, a)$ to $\Omega$. To this end, notice that
\begin{align}\label{eq12}
|\Omega(\omega, a)|^2 &= \left\vert\frac{1-e^{-(a+i\omega)}}{a+i\omega}\right\vert^2 = \frac{2e^{-a}(\cosh a - \cos\omega)}{a^2+\omega^2}.
\end{align}
Employing the statement in Lemma \eqref{lem5.3}, we see that
\begin{align*}
\frac{\cosh a -1}{a^2}\geq\frac12\geq \frac{1-\cos\omega}{\omega^2}, \quad\forall\,a \in\R;\,\forall\,\omega\in \R
\end{align*}
The latter inequality is equivalent to the following expressions:
\begin{align*}
\omega^2 (\cosh a - 1) \geq a^2 & (1 - \cos\omega) \quad\Longleftrightarrow\quad \omega^2 \cosh a - \omega^2 \geq a^2 - a^2\cos\omega\\
& \quad\Longleftrightarrow\quad \omega^2 \cosh a - \omega^2\cos\omega \geq a^2 + \omega^2 - a^2\cos\omega - \omega^2\cos\omega\\
& \quad\Longleftrightarrow\quad \frac{\cosh a - \cos\omega}{a^2+\omega^2} \geq \frac{1 - \cos\omega}{\omega^2}
\end{align*}
Therefore, \eqref{eq12} implies
\begin{equation*}
|\Omega(\omega, a)|^2 = \frac{2e^{-a}(\cosh a - \cos\omega)}{a^2+\omega^2} \geq e^{-a}\frac{2(1 - \cos\omega)}{\omega^2} = e^{-a}\,|\Omega(\omega)|^2.
\end{equation*}

A straightforward computation using again the inequalities in Lemma \ref{lem5.3} shows that
\[
|\Omega(\omega, a)| \leq \frac{1-e^{-a}}{a}, \quad\forall\,a > 0;\,\forall \omega\in \R.
\]
As the right-hand side of the above inequality is bounded above by 1, we obtain an upper bound for $|\Omega(\omega, a)|$ in the form 
\[
|\Omega(\omega, a)| \leq 1 + |\Omega(\omega)|.
\]
These two results are summarized in the next proposition.
\begin{proposition}
For all $a > 0$ and all $\omega\in \R$, the following inequalities hold:
\be
e^{-a/2} |\Omega(\omega)| \leq |\Omega(\omega, a)| \leq 1 + |\Omega(\omega)|.
\ee
\end{proposition}

Next, we use the inequalities in the above proposition to establish lower and upper bounds for $\widehat{E_z^a}$ in terms of $\widehat{B_z}$.
\begin{proposition}\label{prop5.4}
For all $z\in \C_{>1}$ and for all $a > 0$, we have that
\be\label{upperbound}
e^{- a \Re z/2 - 2\pi |\Im z|}\,|\widehat{B_z}| \leq |\widehat{E^a_z}| \leq 1 + 2^{\Re z} e^{2\pi |\Im z|}|\widehat{B_z}|.
\ee
\end{proposition}
\begin{proof}
Let $z\in \C_{>1}$ and $a > 0$. Then, the following estimates hold
\begin{align*}
|\widehat{B_z}| &= |\Omega^z| = |\Omega|^{\Re z}\,e^{-\Im z \Arg\Omega} \leq e^{a\Re z /2} |\Omega(\bullet, a)|^{\Re z}\,e^{-\Im z \Arg\Omega}\\
& = e^{a\Re z /2} |\Omega(\bullet, a)|^{\Re z}\,e^{-\Im z \Arg\Omega(\bullet, a)}\,e^{\Im z (\Arg\Omega(\bullet, a)-\Arg\Omega)}\\
& \leq e^{a\Re z /2} |\widehat{E^a_z}|\,e^{2\pi |\Im z|},
\end{align*}
implying the lower bound.
To verify the upper bound, note that
\begin{align*}
|\widehat{E^a_z}| &= |\Omega(\bullet, a)^z| =  |\Omega (\bullet, a)|^{\Re z}\,e^{-\Im z \Arg\Omega(\bullet, a)}\leq (1 + |\Omega|)^{\Re z}\,e^{-\Im z \Arg\Omega(\bullet, a)}\\
& \leq 1 + 2^{\Re z} |\Omega|^{\Re z}e^{-\Im z \Arg\Omega}\,e^{\Im z (\Arg\Omega - \Arg\Omega(\bullet,a))}\\
& \leq 1 + 2^{\Re z}\,e^{2\pi |\Im z|}\, |\widehat{B_z}|.
\end{align*}
Above, we used the fact that $(1+x)^p \leq 1 + 2^p x^p$, for $0\leq x \leq 1$ and $p\geq 1$.
\end{proof}
The upper bound in \eqref{upperbound} together with the arguments employed in \cite[Theorem 3.1]{unserblu00} and \cite[5.1]{forster06} immediately yield the next result.
\begin{proposition}
The complex exponential B-spline $E^a_z$, $a\geq 0$, is an element of $L^2(\R)$ for $\Re z > \frac12$ and of the Sobolev spaces $W^s(L^2(\R))$ for $\Re z > s + \frac12$.
\end{proposition}

We finish this subsection by mentioning that the complex exponential B-spline $E^a_z$ has a frequency spectrum analogous to that for complex B-splines consisting of a modulating and a damping factor:
$$
|\widehat{E^a_{z}}(\omega)| = |\widehat{E^a_{\Re z}}(\omega)| e^{-i \Im z \ln |\Omega (\omega,a)|} e^{\Im z \arg \Omega(\omega,a)}.
$$
Hence, complex exponential splines combine the advantages of exponential splines as described at the beginning of Section \ref{sec2} with those of complex B-splines as mentioned in Section \ref{sec4}.

In Figure \eqref{fig2}, the graphs of $\widehat{E^a_{z}}$ for $z = 2 + k/4 + i$, $k=0,1,2,3,4$, and $a = 1$ are displayed.
\begin{figure}[h!]
\includegraphics[width = 4cm, height = 3cm]{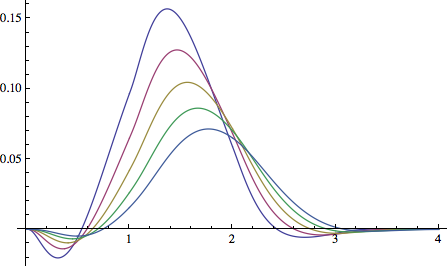}\hspace{2cm}
\includegraphics[width = 4cm, height = 3cm]{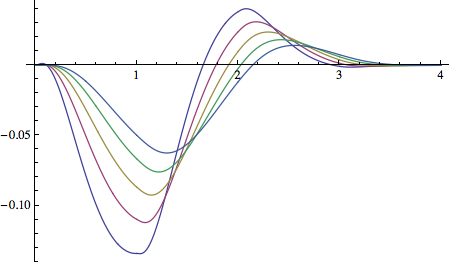}\\
\includegraphics[width = 4cm, height = 3cm]{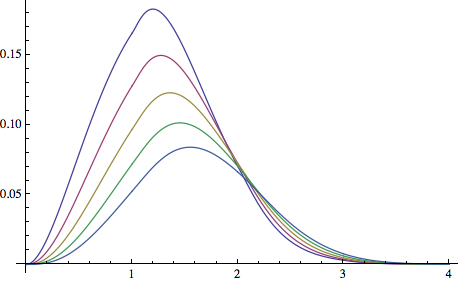}
\caption{The graphs of $\widehat{E^{1}_{z}}$ for $z = 2 + k/4 + i$, $k=0,1,2,3,4$: Real part (upper left), imaginary part (upper right), and modulus (lower middle).}\label{fig2}
\end{figure}

\subsection{Time domain representation}
Next, we derive the time domain representation for a complex exponential B-spline $E^a_z$. For this purpose, we introduce the (backward) exponential difference operator $\nabla_a$ acting of functions $f:\R\to\R$ via
\[
\nabla_a f := f - e^{-a} f(\bullet - 1), \quad a\in \R_0^+.
\] 
For an $n\in \N$, the $n$-fold exponential difference operator is then given by $\nabla_a^n := \nabla_a (\nabla_a^{n-1})$ with $\nabla_a^1 := \nabla_a$. A straightforward calculation yields an explicit formula for $\nabla_a^n$:
\be\label{expdiff}
\nabla_a^n f = \sum_{\ell=0}^\infty \binom{n}{\ell} (-1)^\ell e^{-\ell a} f(\bullet - \ell).
\ee
In the above expression, we replaced the usual upper limit of summation $n$ by $\infty$. This does not alter the value of the sum since for $\ell > n$ the binomial coefficients are identically equal to zero.

Based on the expression \eqref{expdiff}, we define a {\em (backward) exponential difference operator of complex order $\nabla_a^z$} as follows.
\[
\nabla_a^z f := \sum_{\ell=0}^\infty \binom{z}{\ell} (-1)^\ell e^{-\ell a} f(\bullet - \ell), \quad z\in \C_{>1}.
\]

\begin{theorem}
Let $z\in \C_{>1}$. Then the complex exponential B-Spline $E^a_z$ possesses a time domain representation of the form 
\be\label{timerep}
E_a^z (x) = \frac{1}{\Gamma(z)}\,\sum_{\ell=0}^\infty \binom{z}{\ell} (-1)^\ell e^{-\ell a} e_+^{-a(x-\ell)}\,(x-\ell)_+^{z-1},
\ee
where $e_+^{(\bullet)} := \chi_{[0,\infty)}\,e^{(\bullet)}$ and $x_+ := \max\{x,0\}$. The sum converges both point-wise in $\R$ and in the $L^2$--sense.
\end{theorem}

\begin{proof}
For $z\in \C_{>1}$ and $a > 0$, we consider the Fourier transform (in the sense of tempered distributions) of the function $\nabla_a^z e_+^{-a x}x_+^{z-1}$. 
\begin{align*}
\frac{1}{\Gamma (z)}\, (\nabla_a^z e_+^{-a x}x_+^{z-1})^\wedge & = \frac{1}{\Gamma (z)}\,\sum_{\ell=0}^\infty \binom{z}{\ell} (-1)^\ell e^{-\ell a}\,\int_\R e_+^{-a (x - \ell)} (x - \ell)_+^{z-1}\,e^{-i\omega x} dx\\
& = \frac{1}{\Gamma (z)}\,\sum_{\ell=0}^\infty \binom{z}{\ell} (-1)^\ell e^{-\ell a}\,\int_\ell^\infty e^{-a (x - \ell)} (x - \ell)^{z-1}\,e^{-i\omega x} dx\\
& = \frac{1}{\Gamma (z)}\,\sum_{\ell=0}^\infty \binom{z}{\ell} (-1)^\ell e^{-\ell a}\,\int_0^\infty e^{-a x} x^{z-1}\,e^{-i\omega (x + \ell)} dx\\
& = \frac{1}{\Gamma (z)}\,\sum_{\ell=0}^\infty \binom{z}{\ell} (-1)^\ell \,\int_0^\infty x^{z-1}\,e^{-(a + i\omega) (x + \ell)} dx,
\end{align*}
where the interchange of sum and integral is allowed by the Fubini--Tonelli Theorem using the fact that the sum over the binomial coefficients is bounded. (See, for instance, \cite{AS} for the asymptotic behavior of the Gamma function.)

Using the substitution $(a+i\omega)x \mapsto t$, the integral on the left becomes the Gamma function up to a multiplicative factor:
\[
\frac{e^{-(a+i\omega)\ell}}{(a+i\omega)^z}\,\Gamma (z).
\]
Thus,
\[
\frac{1}{\Gamma (z)}\, (\nabla_a^z e_+^{-a x}x_+^{z-1})^\wedge = \sum_{\ell=0}^\infty \binom{z}{\ell} (-1)^\ell \,\frac{e^{-(a+i\omega)\ell}}{(a+i\omega)^z} = \Omega(\omega, a)^z.
\]
Employing a standard density argument, we deduce the validity of the above equality for both the $L^1(\R)$-- and $L^2(\R)$--topology.
\end{proof}
It follows directly from the time representations \eqref{cb} and \eqref{timerep} for complex B-splines, respectively, complex exponential B-splines that
\[
|E^a_z (x)| \leq e_+^{-a x} \, |B_z (x)|, \quad\forall\,x\in \R_0^+;\;\forall\,a\in\R_0^+.
\]
Complex exponential B-splines also satisfy a partition of unity property up to a multiplicative constant:
\[
\int_\R E^a_z (x) dx = \widehat{E^a_z}(0) = \left(\frac{1-e^{-a}}{a}\right)^z \neq 0.
\] 
The graphs of some complex exponential B-splines are shown in Figure \ref{fig3}.
\begin{figure}[h!]
\includegraphics[width = 4cm, height = 2.5cm]{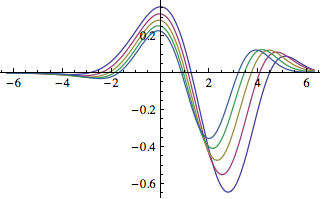}\hspace{2cm}
\includegraphics[width = 4cm, height = 2.5cm]{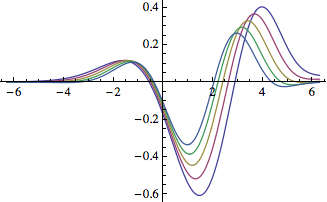}\\
\includegraphics[width = 4cm, height = 2.5cm]{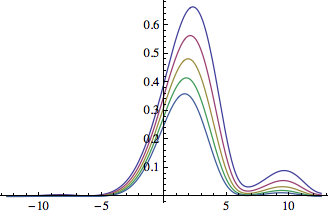}
\caption{The graphs of ${E}^{1.3}_{z}$ for $z=3 + k/4 + i$, $k = 0,1,2,3,4]$: Real part (upper left), imaginary part (upper right), and modulus (lower middle).}\label{fig3}
\end{figure}
\subsection{Multiresolution and Riesz bases}
In this subsection, we investigate multiscale and approximation properties of the complex exponential B-splines. To this end, we consider the relationship between $\widehat{E}^{a}_z (\bullet)$ and $\widehat{E}^{2a}_z (2\,\bullet)$. (The case of real $z > 1$ was first considered in \cite{unserblu05} and then also in \cite{masso}.)

Under the assumptions $z\in \C_{>1}$ and $a > 0$, the following holds:
\begin{align*}
\Omega(2\omega, 2a) & = \frac{1-e^{-(2a+2i\omega)}}{2a+2i\omega} = \frac{(1 + e^{-(a+i\omega)})(1 - e^{-(a+i\omega)})}{2(a+i\omega)}\\
& = \left(\frac{1+e^{-(a+i\omega)}}{2}\right)\,\Omega(\omega,a).
\end{align*}
This then implies that
\[
\widehat{E}^{2a}_z (2\omega) = \left(\frac{1+e^{-(a+i\omega)}}{2}\right)^z\,\widehat{E}^a_z (\omega) =: 2 H_0(\omega, a)\,\widehat{E}^a _z (\omega).
\]
Therefore, the low pass filter $H_0 (\omega,a)$ is given by
\[
H_0(\omega, a) = \frac{1}{2^{z-1}}\,\sum_{k = 0}^\infty \binom{z}{k} e^{- (a+i\omega)k},
\]
from which we immediately derive a two-scale relation between complex exponential B-splines:
\be
E^{2a}_z (x) = \frac{1}{2^{z}}\,\sum_{\ell=0}^\infty \binom{z}{k} e^{-a k}\, E^a_z (2x - k).
\ee

Denote by $T: L^2(\R) \to L^2(\R)$ the unitary translation operator defined by $T f := f(\bullet - 1)$. 
\begin{proposition}
Let $z\in \C_{>1}$ and $a \geq 0$. Then the system $\{T^k E^a_z \st k\in \Z\}$ is a Riesz sequence in $L^2(\R)$.
\end{proposition}

\begin{proof}
It suffices to show that there exist constants $0 < A \leq B < \infty$ so that
\[
A \leq \sum_{k\in \Z} \left\vert \widehat{E}^a_z (\omega + 2\pi k)\right\vert^2 \leq B.
\]
To this end, we use the fact that the complex B-splines form a Riesz sequence of $L^2(\R)$ \cite[Theorem 9]{forster06}, and employ Proposition \ref{prop5.4}.
\end{proof}

\begin{corollary}
Suppose that $z\in \C_{>1}$ and $a\geq 0$. Let
\[
V_0^a := \overline{\Span\{T^k E^a_z\st k\in \Z\}}^{L^2(\R)}.
\]
Then $\{T^k E^a_z \st k\in \Z\}$ is a Riesz basis for $V_0^a$.
\end{corollary}
For $j\in \Z$, we define
\[
V^{2^j a}_j := \overline{\Span\{E^{2^j a}_z (2^j \bullet - k)\st k\in \Z\}}^{L^2(\R)}.
\]
Then
\[
V^{2^j a}_j \subset V^{2^{j+1} a}_{j+1}, \qquad \forall\,j\in \Z.
\]
To establish that the ladder of subspaces $\{V^{2^j a}_j\st j\in \Z\}$ forms a multiresolution analysis of $L^2(\R)$, we use Theorem 2.13 in \cite{Wo}. We have already shown that assumptions (i) (existence of a Riesz basis for $V^a_0$) and (ii) (existence of a two-scale relation) in this theorem hold. The third assumption requires that $\widehat{E^a_z}$ is continuous at the origin and $\widehat{E^a_z}(0) \neq 0$. Both requirements are immediate from \eqref{expspline}. Hence, we arrive at the following result.
\begin{theorem}
Assume that $a\in \R^+_0$ and $z\in \C_{>1}$. Let $\varphi^{2^j a}_{z; j,k} := E^{2^j a}_z (2^j \bullet - k)$. Then the spaces
\[
V^{2^j a}_j := \overline{\Span\{\varphi^{2^j a}_{z; j,k}\st k\in \Z\}}^{L^2(\R)}, \qquad j\in \Z
\]
form a dyadic multiresolution analysis of $L^2(\R)$ with scaling function $\varphi^a_{z;0,0} = E^{a}_z$.
\end{theorem}
Denote the wavelet associated with $E^{a}_z$ by $\theta_z^a$, and the autocorrelation function of $\theta_z^a$ by
\[
R_z^a (\omega) := \sum_{k\in \Z} \vert \widehat{\theta_z^a} (\omega + k)\vert^2.
\]
Then $\widehat{\psi_z^a} := \widehat{\theta_z^a}/\sqrt{R_z^a}$ is an orthonormal wavelet, i.e., $\inn{\widehat{\psi_z^a}}{T_k\widehat{\psi_z^a}}$ $= \delta_{k0}$, $\forall\,k\in \Z$.
\subsection{Connection to fractional differential operators}
Our next goal is to relate complex exponential B-splines to fractional differential operators of type \eqref{diffz} considered in the previous section. For this purpose, we assume throughout this subsection that $a \geq 0$ and $z\in \C_{>1}$. 

As $E_z^a$ is in $L^1_{\loc}$, we see that $E^a_z$ is in the Lizorkin dual $\Psi_+'$. Therefore, $\mcD^z E^a_z$ exists and we can define the following operator. (See also the comment at the end of Section \ref{sec2}.)
\begin{definition}
Let the fractional differential operator $(\mcD + aI)^z:\Psi_+^\prime\to\Psi_+^\prime$ be defined by
\be\label{gendiff}
(\mcD + aI)^z (e^{-a(\bullet)} f) := e^{-a(\bullet)} \mcD^z f.
\ee
\end{definition}
The next results show that the fractional differential operator \eqref{gendiff} satisfies properties similar to those of the associated differential operator of positive integer order.
\begin{proposition}
For the fractional differential operator $(\mcD + aI)^z$ defined in \eqref{gendiff}, the following statements hold. 
\begin{itemize}
\item[(i)] As $f\equiv 1\in \Psi_+^\prime$, the function $e^{-a(\bullet)}\in \ker (\mcD + aI)^z$.
\item[(ii)] The complex monomials $(\bullet)^{z-1}$ are in $\Psi_+^\prime$ implying that $(\bullet)^{z-1}\,e^{-a(\bullet)}\in \ker (\mcD + aI)^z$.
\end{itemize}
Moreover,
\be\label{diffE}
(\mcD + aI)^z E_z^a = \sum_{\ell=0}^\infty \left[\binom{z}{\ell} (-1)^\ell e^{-\ell a}\right] \delta (\bullet - \ell).
\ee
\end{proposition}
\begin{proof}
In order to verify statements (i) and (ii), note that $\inn{\mcD^z 1}{\varphi} = 0$ and $\inn{\mcD^z (\bullet)^{z-1}}{\varphi} = \Gamma (z) \inn{\delta}{\varphi} = \Gamma (z) \varphi(0) = 0$, for all $\varphi\in\Psi_+$. The conclusions now follow from \eqref{gendiff}.
 
Equation \eqref{diffE} is a consequence of (i), (ii), as well as definition \eqref{gendiff} of the operator $(\mcD + aI)^z$, and the fact that $B_z$ satisfies \eqref{cs}.
\end{proof}

Equation \eqref{diffE} suggests a more general definition of exponential spline of complex order.
\begin{definition}
An {\em exponential spline of complex order $z\in \C_{>1}$ corresponding to $a\in \R^+$} is any solution of the fractional differential equation
\be\label{Leq}
(\mcD + a I)^{z} f = \sum_{\ell=0}^\infty c_\ell \,\delta (\bullet -\ell),
\ee
for some $\ell^\infty$-sequence $\{c_\ell\st \ell\in \N\}$.
\end{definition}
Clearly, the complex exponential spline $E^a_z$ is a nontrivial solution of \eqref{Leq}. The coefficients $c_\ell$ are bounded since
\[
\left\vert \sum_{\ell=0}^\infty c_\ell\right\vert = \left\vert\sum_{\ell=0}^\infty\binom{z}{\ell} (-1)^\ell e^{-\ell a}\right\vert \leq \left\vert\sum_{\ell=0}^\infty\binom{z}{\ell}\right\vert \leq c\, e^{|z-1|},
\]
for some constant $c>0$. (See the proof of Theorem 3 in \cite{forster06} for the final inequality.)
\subsection{Generalities}
So far, we have only considered complex exponential B-splines that depend on one parameter $a\in \R^+$. Now we will briefly look at the more general setting based on \eqref{regE}. To this end, let $\boldsymbol{a}:=(a_1, \ldots, a_n)\in (\R_0^+)^n$  be an $n$-tuple of parameters with at least one $a_j\neq 0$.

Let $\bfz:= (z_1, \ldots, z_n)\in \C^n_{> 1} := \underset{j = 1}{\overset{n}{\times}} \C_{>1}$ and define
\[
E_{\bfz}^{\boldsymbol{a}} := \underset{j = 1}{\overset{n}{*}} E_{z_j}^{a_j},
\]
or, equivalently,
\[
\widehat{E_{\bfz}^{\boldsymbol{a}}} := \prod_{j=1}^n \left( \frac{1-e^{-(a_j+i\omega)}}{a_j+i\omega}\right)^{z_j}.
\]
As above, we have that
\[
\int_\R E^{\bfa}_{\bfz} (x) dx = \prod_{j=1}^n \left( \frac{1-e^{-a_j}}{a_j}\right)^{z_j} \neq 0.
\]

For illustrative purposes, let us consider the case $n := 2$, and set $a := a_1$, $b:= a_2$, $z:= z_1$ and $\zeta := z_2$. Using the time domain representation of $E_z^a$ and $E_\zeta^b$, we can compute the time domain representation of the complex exponential B-spline $E_{(z,\zeta)}^{(a,b)}$. The result suggests that there are connections to the theory of special functions. 

By Mertens' Theorem \cite{hardy}, we can write the double product $E_z^a (y) E_\zeta^b (x-y)$ in the following form:
\begin{align*}
E_z^a (y) E_\zeta^b (x-y) &= \frac{1}{\Gamma (z)\Gamma (\zeta)} \sum_{k=0}^\infty \sum_{\ell=0}^k \binom{z}{\ell} (-1)^\ell e^{-\ell a} e_+^{-a (y - \ell)} (y-\ell)_+^{z-1}\\
& \qquad \times \binom{\zeta}{k-\ell} (-1)^{k-\ell} e^{-(k-\ell) b} e_+^{-b (x-y - (k-\ell))} (x-y-(k-\ell))_+^{\zeta-1} \\
& = \frac{1}{\Gamma (z)\Gamma (\zeta)} \sum_{k=0}^\infty \sum_{\ell=0}^k \binom{z}{\ell} \binom{\zeta}{k-\ell} (-1)^k e^{-\ell a - (k-\ell)b} e_+^{-a (y - \ell)} e_+^{-a (y - \ell)}\\
& \qquad \times (y-\ell)_+^{z-1} [(x-k) - (y -\ell)]_+^{\zeta-1}
\end{align*}
Thus,
\begin{align*}
E_{(z,\zeta)}^{(a,b)} (x) &= (E_z^a * E_\zeta^b)(x) = \int_\R E_z^a (y) E_\zeta^b (x-y) dy \\
& = \left(\Sigma\right)\, \int_\R H(y - \ell) e^{-a(y - \ell)} H((x-k)-(y-\ell)) e^{-b((x-k)-(y-\ell))}\\
& \qquad \times (y-\ell)_+^{z-1} [(x-k) - (y -\ell)]_+^{\zeta-1} dy.
\end{align*}
Here, we put all non-variable quantities into the expression $(\Sigma)$ and used the Heaviside function $H$.

Recognizing that $x - k \geq y - \ell \geq 0$ holds, the integral in the last line above, can be written as
\begin{align*}
& = \int_\ell^{x-k+\ell} e^{-a(y-\ell)} e^{b(y - \ell)} e^{-b(x-k)} (y-\ell)^{z-a}  [(x-k) - (y -\ell)]^{\zeta-1} dy\\
& = e^{-b(x-k)} \int_\ell^{x-k+\ell} e^{(a-b)(y-\ell)} (y-\ell)^{z-a}  [(x-k) - (y -\ell)]^{\zeta-1} dy\\
& = e^{-b(x-k)} \int_0^{x-k} e^{-(a-b)\eta} \eta^{z-1}[(x-k) - \eta]^{\zeta-1} d\eta,
\end{align*}
where we used the substitution $y - \ell \mapsto \eta$. After the substitution $\eta\mapsto \tau/(x-k)$, the last integral becomes the integral representation of Kummer's confluent hypergeometric $M$--function \cite{AS}:
\begin{align*}
\int_0^{x-k} e^{-(a-b)\eta} \eta^{z-1} &[(x-k) - \eta]^{\zeta-1} d\eta = \\
& (x-k)^{z + \zeta -1} \frac{\Gamma(z)\Gamma(\zeta)}{\Gamma (z+\zeta)}\,M(z, z+\zeta; -(a-b)(x-k)).
\end{align*}
Combining all terms, we arrive at an explicit formula for $E_{(z,\zeta)}^{(a,b)}$:
\begin{align}\label{mess}
E_{(z,\zeta)}^{(a,b)} (x) &= \frac{1}{\Gamma (z+\zeta)} \sum_{k=0}^\infty\,\left[ \sum_{\ell=0}^k \binom{z}{\ell} \binom{\zeta}{k-\ell}\,e^{-\ell(a-b)}\right] (-1)^k\,e^{-b x}\nonumber\\
& \qquad\times M(z, z+\zeta; -(a-b)(x-k))\,(x-k)^{z+\zeta-1}.
\end{align}
Realizing that the expression in brackets is equal to \cite{AS}
\[
\binom{\zeta}{k} {}_2F_1 (-k, -z,1-k+\zeta; e^{-(a-b)}),
\]
we may write \eqref{mess} also as
\begin{align*}
E_{(z,\zeta)}^{(a,b)} (x) &= \frac{1}{\Gamma (z+\zeta)} \sum_{k=0}^\infty\,\binom{\zeta}{k} (-1)^k\,e^{-b x}\, {}_2F_1 (-k, -z,1-k+\zeta; e^{-(a-b)})\\
& \qquad\qquad\times M(z, z+\zeta; -(a-b)(x-k))\,(x-k)^{z+\zeta-1}.
\end{align*}
Notice that the above equation is a sampling procedure involving Kummer's $M$-function (and Gau{\ss}'s ${}_2 F_1$ hypergeometric function).
\section{Summary and Further Work}
We extended the concept of exponential B-spline to complex orders $z\in \C_{>1}$. This extension contains as a special case the class of polynomial splines of complex order. The new class of complex exponential B-splines generates multiresolution analyses of and wavelet bases for $L^2(\R)$, and relates to fractional differential operators defined on Lizorkin spaces and their duals. 

Explicit formulas for the time domain representation of complex exponential B-spline depending on one parameter and two parameters were derived. In the latter case, there seem to be connection to the theory of special functions as the Kummer $M$-function appears in the representation.

An approximation-theoretic investigation of complex exponential B-splines for several parameters needs to be initiated and numerical schemes for the associated approximation spaces developed. Connections to fractional differential operators of the form 
\[
\mathcal{L}^{\bfz}_{\bfa} := \prod_{i=1}^n (\mcD + a_i I)^{z_i},
\]
where $\bfa = (a_1, \ldots, a_n)\in (\R^+)^n$ and $\bfz = (z_1, \ldots, z_n)\in \C_{>1}^n$, need to be established. Moreover, the time domain representation for complex exponential B-splines depending on an $n$-tuple $\bfa\in (\R^+)^n$ of parameters has to be derived. Finally, the relation to special functions is worthwhile an investigation.
\bibliographystyle{amsalpha}
\bibliography{exponential_complex}

\end{document}